\documentclass[a4paper,12pt]{article}

\usepackage{amsmath}
\usepackage{amsthm}
\usepackage{amssymb}

\usepackage{epsfig}
\usepackage{psfrag}
\usepackage{subfigure}
\usepackage{graphicx}

\usepackage[mathscr,mathcal]{eucal}
\usepackage{enumerate}

\usepackage{t1enc}
\usepackage[latin2]{inputenc}

\def \N {\mathbb N}

\def \R {\mathbb R}

\def \Z {\mathbb Z}

\def\boP{\mathbf{P}}

\def\cF{\mathcal{F}}

\def\cT{\mathcal{T}}

\def\cX{\mathcal{X}}

\def\vareps{\varepsilon}

\def\phi{\varphi}

\newcommand{\prob}[1]{\ensuremath{\mathbf{P}\,\big(#1\big)}}
\newcommand{\expect}[1]{\ensuremath{\mathbf{E}\,\big(#1\big)}}

\newcommand{\condprob}[2]{\ensuremath{\mathbf{P}\,\big(#1\mid#2\big)}}
\newcommand{\condexpect}[2]{\ensuremath{\mathbf{E}\,\big(#1\mid#2\big)}}

\newcommand{\ind}[1]{\ensuremath{{1\!\!1}{\{#1\}}}}

\def \toprob {\,\,\buildrel\boP\over\longrightarrow\,\,}

\def\la{\langle}
\def\ra{\rangle}

\renewcommand{\d}{\,\mathrm d}

\def \wt {\widetilde}

\newtheorem {theorem}{Theorem}
\newtheorem {thm}[theorem]{Theorem}
\newtheorem {lemma}{Lemma}

\newtheorem {cor}{Corollary}
\newtheorem {proposition}{Proposition}
\newtheorem {prop}{Proposition}

\newtheorem* {theorem*}{Theorem}
\newtheorem* {thm*}{Theorem}
\newtheorem* {lemma*}{Lemma}
\newtheorem* {lem*}{Lemma}
\newtheorem* {corollary*}{Corollary}
\newtheorem* {cor*}{Corollary}
\newtheorem* {proposition*}{Proposition}
\newtheorem* {prop*}{Proposition}
\newtheorem* {definition*}{Definition}
\newtheorem* {def*}{Definition}
\newtheorem* {remark*}{Remark}
\newtheorem* {rem*}{Remark}

\def\be{\begin{equation}}
\def\ee{\end{equation}}

\def\bea{\begin{eqnarray}}
\def\eea{\end{eqnarray}}

\DeclareMathOperator{\UNI}{UNI}

\def \pell {\ell^+}
\def \mell {\ell^-}
\def \pmell{\ell^{\pm}}

\title{Self-repelling random walk with directed edges on $\Z$}

\author{
{\sc B\'alint T\'oth} \ \ \  and \ \ \ {\sc B\'alint Vet{\H o}}
\\
Institute of Mathematics
\\
Budapest University of Technology
}

\begin{document}

\maketitle

\centerline
{\sl Dedicated to J\'ozsef Fritz on the occasion of
his 65$^{\text{th}}$ birthday.}

\vskip1cm

\begin{abstract}
We consider a variant of self-repelling random walk on the integer lattice
$\Z$ where the self-repellence is defined in terms of the local time on
\emph{oriented} edges. The long-time asymptotic scaling of this walk is
surprisingly different from the asymptotics of the similar process with
self-repellence defined in terms of local time on unoriented edges, examined in
\cite{toth_95}. We prove limit theorems for the local time process and for the
position of the random walker. The main ingredient is a Ray\,--\,Knight-type of
approach. At the end of the paper, we also present some computer simulations
which show the strange scaling behaviour of the walk considered.
\end{abstract}

\section{Introduction}
\label{s:intro}

The \emph{true self-avoiding random walk} on
$\Z$
is a nearest neighbour random
walk, which is locally pushed in the direction of the negative
gradient of its
own local time (i.e.\ occupation time measure). For precise
formulation and
historical background, see
\cite{amit_parisi_peliti_83},
\cite{peliti_pietronero_87},
\cite{obukhov_peliti_83},
\cite{toth_95},
the survey papers
\cite{toth_99},
\cite{toth_01a},
and/or further references cited there. In
\cite{toth_95},
the edge version of the problem was considered, where the walk is
pushed  by the negative gradient of the local time spent on
\emph{unoriented}  edges. There, precise asymptotic limit theorems
were proved for the   local time process and position of the random
walker at late times, under space scaling proportional to the 2/3-rd
power of time. For a survey of these and related results, see
\cite{toth_99},
\cite{toth_01a},
\cite{pemantle_07}.
Similar results for the site version have been obtained recently,
\cite{toth_veto_08b}.
In the present paper, we consider a similar problem but with the walk
being pushed by the local differences of occupation time measures on
\emph{oriented} rather than unoriented edges. The behaviour is
phenomenologically surprisingly  different from the unoriented case:
we prove limit theorems under square-root-of-time (rather than
time-to-the-$\frac23$) space-scaling but the limit laws are not the
usual diffusive ones. Our model belongs to the wider class of
\emph{self-interacting   random walks} which attracted attention in
recent times, see e.g.
\cite{norris_rogers_williams_87},
\cite{durrett_rogers_92},
\cite{cranston_mountford_96},
\cite{toth_werner_98},
\cite{mountford_tarres_08}
for a few other examples. In all these cases long memory of the random
walk or  diffusion is induced by a self-interaction mechanism defined
locally in a natural way in terms of the local time (or occupation
time) process. The main  challenge is to understand the asymptotic
scaling limit (at late times) of the process.

Let
$w$
be a weight function which is non-decreasing and non-constant:
\begin{equation}
\label{growth}
w:\Z\to\R_+,
\qquad
w(z+1)\ge w(z),
\qquad
\lim_{z\to\infty}\big(w(z)-w(-z)\big)>0.
\end{equation}
We will consider a nearest neighbour random walk
$X(n)$,
$n\in\Z_+:=\{0,1,2,\dots\}$,
on the integer lattice
$\Z$,
starting from
$X(0)=0$,
which is governed by its local time process through  the function
$w$
in the following way. Denote by
$\pmell(n,k)$, $(n,k)\in\Z_+\times\Z$,
the local time (that is: its occupation time measure) on oriented
edges:
\begin{equation*}
\pmell(n,k)
:=
\#\{0\le j\le n-1 \,:\, X(j)=k,  \ \ X(j+1)=k\pm1\},
\end{equation*}
where
$\#\{\dots\}$
denotes cardinality of the set. Note that
\begin{equation}
\label{leq}
\pell(n,k)-\mell(n,k+1)=\left\{
\begin{array}{rl}
+1&\text{if}\quad 0\le k < X(n),
\\[8pt]
-1&\text{if}\quad X(n)\le k < 0,
\\[8pt]
0 &\text{otherwise.}
\end{array}
\right.
\end{equation}
We will also use the notation
\begin{equation}
\label{unor}
\ell(n,k)
:=
\pell(n,k)+\mell(n,k+1)
\end{equation}
for the local time spent on the unoriented edge
$\la k,k+1 \ra$.

Our random walk is governed by the  evolution rules
\begin{eqnarray}
\notag
\lefteqn{\condprob{X(n+1)=X(n)\pm1} {{\cF}_n}=}
\\[8pt]
\label{walktransprob}
&&
=
\frac{w(\mp(\pell(n,X(n))-\mell(n,X(n))))}
{w(\pell(n,X(n))-\mell(n,X(n)))+w(\mell(n,X(n))-\pell(n,X(n)))},
\\[8pt]
\notag
\lefteqn{
\pmell(n+1,k)
=
\pmell(t,x) + \ind{X(n)=k, \ \ X(n+1)=k\pm1}.
}
\end{eqnarray}
That is: at each step, the walk prefers to choose that oriented edge
pointing away from the actually occupied site which had been less
visited in the past. In this way balancing or smoothing out the
roughness of the occupation  time measure. We prove limit theorems for
the local time process and for the position of the random walker at
large times under \emph{diffusive   scaling}, that is: essentially for
$n^{-1/2}\ell(n,\lfloor n^{1/2}x\rfloor)$
and
$n^{-1/2}X(n)$,
but with limit laws strikingly different from usual diffusions. See
Theorem
\ref{thmSconv}
and
\ref{thmprocconv}
for precise statement.

The paper is further organized as follows. In Section \ref{s:results},
we formulate the main results. In Section \ref{s:ltconvproof}, we
prove Theorem \ref{thmSconv} about the convergence in sup-norm and in
probability of the local time process stopped at inverse local
times. As a consequence, we also prove convergence in probability of
the inverse local times to \emph{deterministic values}. In Section
\ref{s:procconvproof}, we convert the limit  theorems for the inverse
local times to local limit theorems for the position  of the random
walker at independent random stopping times of geometric distribution
with large expectation. Finally, in Section \ref{s:simulations}, we
present some numerical simulations of the position and local time
processes with particular choices of the weight function
$w(k)=\exp(\beta k)$.

\section{The main results}
\label{s:results}

As in
\cite{toth_95},
the key to the proof is a Ray\,--\,Knight-approach. Let
\begin{equation*}
\label{defT}
T^{\pm}_{j,r}
:=
\min\{n\ge0:\pmell(n,j)\ge r\},
\qquad
j\in\Z,
\quad
r\in\Z_+
\end{equation*}
be the so called inverse local times and
\begin{equation}
\label{Lam}
\Lambda^{\pm}_{j,r}(k)
:=
\ell(T^{\pm}_{j,r},k)
=
\ell^+(T^{\pm}_{j,r},k)
+
\ell^-(T^{\pm}_{j,r},k+1),
\qquad
j,k\in\Z,
\quad
r\in\Z_+
\end{equation}
the local time sequence (on unoriented edges)
of the walk stopped at the inverse local
times. We denote by
$\lambda^{\pm}_{j,r}$
and
$\rho^{\pm}_{j,r}$
the leftmost,  respectively, the  rightmost edges visited by the walk
before the stopping time
$T^{\pm}_{j,r}$:
{\setlength{\arraycolsep}{.13889em}
\begin{eqnarray*}
\lambda^{\pm}_{j,r}
&:=&
\inf\{k\in\Z \,:\,
\Lambda^{\pm}_{j,r}(k)>0\},
\\[8pt]
\rho^{\pm}_{j,r}
&:=&
\sup\{k\in\Z \,:\,
\Lambda^{\pm}_{j,r}(k)>0\}.
\end{eqnarray*}}

The next proposition states that the random walk is recurrent in the
sense that it visits infinitely often every site and edge of
$\Z$.

\begin{prop}
\label{propfin}
Let
$l\in\Z$
and
$m\in\Z_+$
be fixed. We have
\begin{equation*}
\max\left\{
T^{\pm}_{j,r}
\hskip2mm,
\hskip2mm
\rho^{\pm}_{j,r}-\lambda^{\pm}_{j,r}\hskip2mm,
\hskip2mm
\sup_{k}\Lambda^{\pm}_{j,r}(k)
\right\}<\infty
\end{equation*}
almost surely.
\end{prop}

Actually,
we will see it from the proofs of our theorems that the quantities in
Proposition \ref{propfin} are finite, and much stronger results are
true for them, so we do not give a separate proof of this statement.

\subsection{Limit theorem for the local time process}
\label{ss:ltconv}

The main result concerning the local time process stopped at inverse
local times is the following:

\begin{thm}
\label{thmSconv}
Let
$x\in\R$
and
$h\in\R_+$
be fixed. Then
{\setlength{\arraycolsep}{.13889em}
\begin{eqnarray}
\label{lambdaconv}
A^{-1} \lambda^{\pm}_{\lfloor Ax\rfloor,\lfloor Ah\rfloor}
& \toprob &
-|x|-2h,
\\[8pt]
\label{rhoconv}
A^{-1} \rho^{\pm}_{\lfloor Ax\rfloor,\lfloor Ah\rfloor}
& \toprob &
\phantom{-}
|x|+2h,
\end{eqnarray}}
and
\begin{equation}
\label{Lambdaconv}
\sup_{y\in\R}
\left|
A^{-1} \Lambda^{\pm}_{\lfloor Ax\rfloor,\lfloor Ah\rfloor} (\lfloor Ay\rfloor) -
{\left(|x|-|y|+2h\right)}_+
\right|
\toprob
0
\end{equation}
as
$A\to\infty$.
\end{thm}

Note that
\begin{equation*}
T^\pm_{j,r}
=
\sum_{k=\lambda^\pm_{j,r}}^{\rho^\pm_{j,r}}
\Lambda^\pm_{j,r}(k).
\end{equation*}
Hence, it follows immediately from Theorem \ref{thmSconv} that

\begin{cor}
\label{corTconv}
With the notations of Theorem \ref{thmSconv},
\begin{equation}
\label{Tconv}
A^{-2}T^{\pm}_{\lfloor Ax\rfloor,\lfloor Ah\rfloor}
\toprob
(|x|+2h)^2
\end{equation}
as $A\to\infty$.
\end{cor}

Theorem \ref{thmSconv} and Corollary \ref{corTconv} will be proved in
Section \ref{s:ltconvproof}.

\medskip
\noindent
{\bf Remark:}
Note that the local time process and the inverse local times converge
in probability to deterministic  objects rather than converging weakly
in distribution to genuinely random variables. This makes the present
case somewhat similar to the weakly reinforced random walks studied in
\cite{toth_97}.

\subsection{Limit theorem for the position of the walker}
\label{ss:procconv}

According to the arguments in
\cite{toth_95},
\cite{toth_99},
\cite{toth_01a},
from the limit theorems
\begin{equation*}
A^{-1/\nu}T^{\pm}_{\lfloor Ax\rfloor,\lfloor A^{(1-\nu)/\nu}h\rfloor}
\Rightarrow
\cT_{x,h}
\end{equation*}
valid for any
$(x,h)\in\R\times\R_+$,
one can essentially derive the limit theorem for the one-dimensional
marginals of the position process:
\begin{equation*}
A^{-\nu}X(\lfloor At\rfloor)
\Rightarrow
\cX(t).
\end{equation*}
Indeed, the summation arguments, given in detail in the papers quoted
above, indicate that
\begin{equation*}
\varphi(t,x)
:=
2\frac{\partial}{\partial t} \int_0^\infty \prob{\cT_{x,h}<t} \d h
\end{equation*}
is the good candidate for the the density of the distribution of
$\cX(t)$,
with respect to Lebesgue-measure. The scaling relation
\begin{equation}
\label{piscaling}
A^{1/\nu}\varphi(At,A^{1/\nu}x)
=
\varphi(t,x)
\end{equation}
clearly holds. In some cases (see e.g.\ \cite{toth_95}) it is not trivial to
check that $x\mapsto\varphi(t,x)$ is a bona fide probability density of total
mass $1$. (However, a Fatou-argument easily shows that its total mass is not
more than 1.) But in our present case, this fact drops out from explicit
formulas. Indeed, the weak limits \eqref{Tconv} hold, which, by straightforward
computation, imply
\begin{equation*}
\varphi(t,x)
=
\frac1{2\sqrt t}\ind{|x| \le \sqrt t}.
\end{equation*}

Actually, in order to prove limit theorem for the position of the
random walker, some smoothening in time is needed,  which is realized
through the Laplace-transform. Let
\begin{equation*}
\hat\varphi(s,x)
:=
s\int_0^\infty e^{-st}\varphi(t,x)\d t
=
\sqrt{s\pi}(1-F(\sqrt{2s}|x|))
\end{equation*}
where
\begin{equation*}
F(x):=\frac{1}{\sqrt{2\pi}}
\int_{-\infty}^x e^{-y^2/2}\d y
\end{equation*}
is the standard normal distribution function.

We prove the following \emph{local limit theorem} for the position of
the random walker stopped at an independent geometrically distributed
stopping time of large expectation:

\begin{thm}
\label{thmprocconv}
Let
$s\in\R_+$
be fixed and
$\theta_{s/A}$
a random variable with geometric distribution
\begin{equation}
\label{thetadistr}
\prob{\theta_{s/A}=n}
=
(1-e^{-s/A})e^{-sn/A}
\end{equation}
which is independent of the random walk
$X(n)$.
Then, for almost all
$x\in\R$,
\begin{equation*}
A^{1/2}\prob{X(\theta_{s/A})
=
\lfloor A^{1/2}x\rfloor}\to\hat\varphi(s,x)
\end{equation*}
as
$A\to\infty$.
\end{thm}

From the above local limit theorem, the integral limit theorem follows
immediately:
\begin{equation*}
\lim_{A\to\infty}
\prob{A^{-1/2}X(\theta_{s/A})<x}
=
\int_{-\infty}^x
\hat\varphi(s,y)\d y.
\end{equation*}

From \eqref{lambdaconv} and \eqref{rhoconv}, the tightness of the
distributions $(A^{-1/2}X(\lfloor At\rfloor))_{A\ge1}$ follows easily. Theorem
\ref{thmprocconv} yields that if the random walk
$X(\cdot)$
has any scaling limit, then
\begin{equation}
\label{detlim}
A^{-1/2}X(\lfloor At\rfloor)
\Longrightarrow
\UNI(-\sqrt t,\sqrt t)
\end{equation}
as
$A\to\infty$
holds where
$\UNI(-\sqrt t, \sqrt t)$
stands for the uniform distribution on the interval
$(-\sqrt t, \sqrt t)$.

The proof of Theorem \ref{thmprocconv} is presented in Section
\ref{s:procconvproof}.

\section{Proof of Theorem \ref{thmSconv}}
\label{s:ltconvproof}

The proof is organized as follows. We introduce independent auxiliary
Markov-chains associated to the vertices of $\Z$ in such a way that
the value of the local time at the edges can be expressed with a sum
of such Markov-chains. It turns out that the auxiliary Markov-chains
converge exponentially fast to their common unique stationary
distribution. It allows us to couple the local time process of the
self-repelling random walk with the sum of i.i.d.\ random
variables. The coupling yields that the law of large numbers for
i.i.d.\ variables can be applied for the behaviour of the local time,
with high probability. The coupling argument breaks down when the
local time approaches
$0$.
We show in Subsection \ref{ss:hit0}, how to handle this case.

Let
\begin{equation}
\label{Ldef}
L_{j,r}(k)
:=
\pell(T_{j,r}^+,k).
\end{equation}
Mind that due to \eqref{leq}, \eqref{unor} and \eqref{Lam}
\begin{equation}
\label{kisdiff}
\big|
\Lambda^+_{j,r}(k)-2L_{j,r}(k)
\big|
\le1.
\end{equation}
We give the proof of \eqref{lambdaconv}, \eqref{rhoconv} and
\begin{equation*}
\sup_{y\in\R}
\left|
A^{-1}L_{\lfloor Ax\rfloor,\lfloor Ah\rfloor}(\lfloor Ay\rfloor)
-
\left(\frac{|x|-|y|}{2}+h\right)_+
\right|
\toprob
0
\end{equation*}
as
$A\to\infty$,
which, due to \eqref{kisdiff}, implies \eqref{Lambdaconv} for
$\Lambda^+$.
The case of
$\Lambda^-$
can be done similarly. Without loss of generality, we can suppose that
$x\le0$.

\subsection{Auxiliary Markov-chains}
\label{subs:aux}

First we define the $\Z$-valued Markov-chain
$l\mapsto\xi(l)$ with the
following transition probabilities:
{\setlength{\arraycolsep}{.13889em}
\begin{align}
\label{defp}
\condprob{\xi(l+1)=x+1}{\xi(l)=x}
&=
\frac{w(-x)}{w(x)+w(-x)}
=:
p(x),
\\[1ex]
\label{defq}
\condprob{\xi(l+1)=x-1}{\xi(l)=x}
&=
\frac{w(x)}{w(x)+w(-x)}
=:
q(x).
\end{align}}

Let
$\tau_{\pm}(m)$, $m=0,1,2,\dots$
be the stopping times of consecutive
upwards/downwards steps of
$\xi$:
\begin{equation*}
\tau_\pm(0):=0,
\qquad
\tau_\pm(m+1)
:=
\min\left\{l>\tau_\pm(m):\xi(l)=\xi(l-1)\pm1\right\}.
\end{equation*}
Then, clearly, the processes
\begin{equation*}
\label{etadef}
\eta_+(m)
:=
-\xi(\tau_+(m)),
\qquad
\eta_-(m)
:=
+\xi(\tau_-(m))
\end{equation*}
are themselves Markov-chains on $\Z$.
Due to the $\pm$ symmetry of the process $\xi$, the Markov-chains $\eta_+$
and $\eta_-$ have the same law. In the present subsection, we simply
denote them by $\eta$ neglecting the subscripts $\pm$.
The transition probabilities of this process are
\begin{equation}
\label{defP}
P(x,y)
:=
\condprob{\eta(m+1)=y}{\eta(m)=x}
=
\left\{
\begin{array}{ll}
\prod_{z=x}^y p(z)q(y+1)
&
\mbox{if}\quad y\ge x-1,
\\[8pt]
0
&
\mbox{if}\quad y<x-1.
\end{array}\right.
\end{equation}

In the following lemma, we collect the technical ingredients of the
forthcoming proof of our limit theorems. We identify  the stationary
measure of the Markov-chain $\eta$, state exponential tightness of the
distributions of
$\big(\eta(m)\,\big|\,\eta(0)=0\big)$ uniformly in $m$
and exponentially fast convergence to
stationarity.

\begin{lemma}
\label{expconvlemma}
\begin{enumerate}[(i)]
\item
The unique stationary measure of the Markov-chain
$\eta$ is
\begin{equation}
\label{defrho}
\rho(x)
=
Z^{-1}
\prod_{z=1}^{\lfloor|2x+1|/2\rfloor}
\frac{w(-z)}{w(z)}
\quad\text{with}\quad
Z:=
2\sum_{x=0}^\infty
\prod_{z=1}^x\frac{w(-z)}{w(z)}.
\end{equation}
\item
There exist constants 
$C<\infty$, $\beta>0$ 
such that for all
$m\in\N$
and 
$y\in\Z$
\begin{equation}
\label{expotight}
P^m(0,y)\le C e^{-\beta |y|}.
\end{equation}
\item
There exist constants
$C<\infty$
and
$\beta>0$
such that for all
$m\ge0$
\begin{equation}
\label{expconv}
\sum_{y\in\Z}
\left|
P^m(0,y)-\rho(y)
\right|
<
C e^{-\beta m}.
\end{equation}
\end{enumerate}
\end{lemma}

\medskip
\noindent
{\bf Remark on notation:}
We shall use the generic notation
\begin{equation*}
\text{\tt something}  \le C e^{-\beta Y}
\end{equation*}
for exponentially strong bounds. The constants $C<\infty$ and $\beta>0$ will
vary at different occurrences and they may (and will) depend on various fixed
parameters but of course not on quantities appearing in the expression $Y$.
There will be no cause for confusion.

\medskip

Note that for any choice of the weight function $w$
\begin{equation}
\label{mean1/2}
\sum_{x=-\infty}^{+\infty}x\rho(x)=-\frac12.
\end{equation}

\begin{proof}
[Proof of Lemma \ref{expconvlemma}]
The following proof is reminiscent of the proof of Lemmas 1 and 2
from \cite{toth_95}. It is somewhat streamlined and weaker conditions
are assumed.

(i) The irreducibility of the Markov-chain $\eta$ is straightforward. One can
easily rewrite \eqref{defP}, using \eqref{defrho}, as
\begin{equation*}
P(x,y)
=
\left\{
\begin{array}{ll}
\frac1{\rho(x)}\left(p(x)\prod_{z=x+1}^{y+1}
q(z)\right)\rho(y) & \mbox{if}\quad y\ge x-1,
\\[8pt]
0 & \mbox{if}\quad y<x-1.
\end{array}\right.
\end{equation*}
It yields that
$\rho$
is indeed stationary distribution for
$\eta$,
because
\begin{equation*}
\sum_{x\in\Z}\rho(x)P(x,y)
=
\left(\sum_{x\le y+1}
p(x)\prod_{z=x+1}^{y+1}
q(z)\right)\rho(y)=\rho(y)
\end{equation*}
where the last equality holds, because
$\lim_{z\to-\infty}\prod_{u=z}^{y+1}q(u)=0$.

(ii)
The stationarity of
$\rho$
implies that
\begin{equation}
\label{Pest}
P^n(0,y)
\le
\frac{\rho(y)}{\rho(0)}
=
\prod_{z=1}^{\lfloor|2y+1|/2\rfloor}\frac{w(-z)}{w(z)}
\le
Ce^{-\beta |y|}.
\end{equation}
The exponential bound follows from \eqref{growth}.
As a consequence, we get finite expectations
in the forthcoming steps of the proofs below.

(iii)
Define the stopping times
{\setlength{\arraycolsep}{.13889em}
\begin{eqnarray*}
\theta_+&=\min\{n\ge0:\eta(n)\ge0\},
\\[8pt]
\theta_0\,&=\min\{n\ge0:\eta(n)=0\}.
\end{eqnarray*}}

From Theorem 6.14 and Example 5.5(a) of
\cite{nummelin_84},
we can conclude the
exponential convergence \eqref{expconv}, if for some $\gamma>0$
\begin{equation}
\label{Eexpsigma}
\condexpect{\exp(\gamma\theta_0)}{\eta(0)=0}<\infty
\end{equation}
holds.

The following decomposition is true, because the Markov-chain
$\eta$
can jump at most one step to the left.
\begin{equation}
\label{sigmafelb}
\begin{split}
&
\condexpect{\exp(\gamma\theta_0)}{\eta(0)=0}
=
e^{\gamma}\sum_{y\ge0} P(0,y)
\condexpect{\exp(\gamma\theta_0)}{\eta(0)=y}
\\[8pt]
&
+e^{\gamma}P(0,-1)\sum_{y\ge0}
\condexpect{\exp(\gamma\theta_+)\ind{\eta(\theta_+)=y}}{\eta(0)=-1}
\condexpect{\exp(\gamma\theta_0)}{\eta(0)=y}.
\end{split}
\end{equation}

One can easily check that, given
$\eta(0)=-1$,
the random variables
$\theta_+$ and $\eta(\theta_+)$
are independent, and for
$y\ge0$
\begin{equation}
\label{thetafelb}
\condexpect{\exp(\gamma\theta_+)
\ind{\eta(\theta_+)=y}} {\eta(0)=-1}
=
\frac{P(0,y)}{1-P(0,-1)}
\condexpect{\exp(\gamma\theta_+)}{\eta(0)=-1}.
\end{equation}

Combining \eqref{sigmafelb} and \eqref{thetafelb} gives us
\begin{equation}
\label{sigmaatiras}
\begin{split}
&
\condexpect{\exp(\gamma\theta_0)}{\eta(0)=0}
\\[8pt]
&
=
e^{\gamma} \sum_{y\ge0} P(0,y)
\condexpect{\exp(\gamma\theta_0)}{\eta(0)=y}
\left(
1+\frac{P(0,-1)}{1-P(0,-1)}
\condexpect{\exp(\gamma\theta_+)}{\eta(0)=-1}
\right).
\end{split}
\end{equation}

So, in order to get the result, we need to prove that for properly
chosen $\gamma>0$
\begin{equation}
\label{thetaplusbound}
\condexpect{\exp(\gamma\theta_+)}{\eta(0)=-1}<\infty
\end{equation}
and
\begin{equation}
\label{thetazerobound}
\condexpect{\exp(\gamma\theta_0)}{\eta(0)=y}\le C e^{\frac{\beta}{2} y}
\qquad
\text{ for }y\in\Z_+
\end{equation}
where $\beta$ is the constant in \eqref{expotight}.

In order to make the argument shorter, we make the assumption
\begin{equation*}
w(-1)<w(+1),
\end{equation*}
or, equivalently,
\begin{equation*}
p(1)=\frac{w(-1)}{w(+1)+w(-1)}
<\frac12<
\frac{w(+1)}{w(+1)+w(-1)}=q(1).
\end{equation*}
The proof can be easily extended for the weaker assumption \eqref{growth}, but
the argument is somewhat longer.

First, we prove \eqref{thetaplusbound}.
Let $x<0$ and $x-1\le y< 0$. Then the following stochastic
domination holds:
\begin{equation}
\label{stdom}
\sum_{z\ge y} P(x,z)
=
\prod_{z=x}^y p(z)
\ge
p(-1)^{y-x+1}
=
q(1)^{y-x+1}.
\end{equation}
Let $\zeta(r)$, $r=1,2,\dots$ be i.i.d.\ random variables
with geometric law:
\begin{equation*}
\prob{\zeta = z} = q(1)^{z+1}p(1),
\quad
z=-1,0,1,2,\dots,
\end{equation*}
and
\begin{equation*}
\wt\theta
:=
\min
\big\{
t\ge 0: \sum_{s=1}^t \zeta(s)\ge 1
\big\}.
\end{equation*}
Note that $\expect{\zeta}>0$.
From the stochastic domination \eqref{stdom}, it follows that for any 
$t\ge0$
\begin{equation*}
\condprob{\theta_+>t}{\eta(0)=-1}
\le
\prob{\wt\theta>t},
\end{equation*}
and hence
\begin{equation*}
\condexpect{\exp(\gamma\theta_+)}{\eta(0)=-1}
\le
\expect{\exp(\gamma\wt\theta)}
<
\infty
\end{equation*}
for sufficiently small $\gamma>0$.

Now, we turn to \eqref{thetazerobound}.
Let now $0\le x-1 \le y$. In this case, the following stochastic
domination is true:
\begin{equation}
\label{stdom2}
\sum_{z\ge y} P(x,z)
=
\prod_{z=x}^yp(z)\le p(1)^{y-x+1}.
\end{equation}
Let now $\zeta(r)$, $r=1,2,\dots$ be i.i.d.\ random variables
with geometric law:
\begin{equation*}
\prob{\zeta = z} = p(1)^{z+1}q(1),
\quad
z=-1,0,1,2,\dots,
\end{equation*}
and for $y\ge 0$
\begin{equation*}
\wt\theta_y
:=
\min
\big\{
t\ge 0: \sum_{s=1}^t \zeta(s)\le -y
\big\}.
\end{equation*}
Note that now $\expect{\zeta}<0$.
From the stochastic domination \eqref{stdom2}, it follows now that with
$y\ge0$, for any $t\ge0$
\begin{equation*}
\condprob{\theta_0>t}{\eta(0)=y}
\le
\prob{\wt\theta_y>t},
\end{equation*}
and hence
\begin{equation*}
\condexpect{\exp(\gamma\theta)} {\eta(0)=y}
\le
\expect{\exp(\gamma\wt\theta_y)}
\le
Ce^{\frac{\beta}{2}y},
\end{equation*}
for sufficiently small $\gamma>0$.
\end{proof}

\subsection{The basic construction}

For $j\in\Z$, denote the inverse local times
(times of jumps leaving site $j\in\Z$)
\begin{equation}
\label{gamdef}
\gamma_j(l)
:=
\min\left\{n:\pell(n,j)+\mell(n,j)\ge l\right\},
\end{equation}
and
\begin{align}
\label{xijdef}
&
\xi_j(l)
:=
\pell(\gamma_j(l),j)-\mell(\gamma_j(l),j),
\\[8pt]
\label{taujdef}
&
\tau_{j,\pm}(0):=0,
\qquad
\tau_{j,\pm}(m+1)
:=
\min\left\{l>\tau_{j,\pm}(m):\xi_j(l)=\xi_j(l-1)\pm1\right\},
\\[8pt]
\label{etajdef}
&
\eta_{j,+}(m)
:=
-\xi_j(\tau_{j,+}(m)),
\qquad
\eta_{j,-}(m)
:=
+\xi_j(\tau_{j,-}(m)).
\end{align}

The following proposition is the key to the Ray\,--\,Knight-approach.

\begin{proposition}
\begin{enumerate}[(i)]
\item
The processes $l\mapsto\xi_j(l)$, $j\in\Z$, are independent copies of the
Markov-chain $l\mapsto\xi(l)$, defined in Subsection \ref{subs:aux}, starting
with initial conditions $\xi_j(0)=0$.
\item
As a consequence: the processes $k\mapsto\eta_{j,\pm}(k)$, $j\in\Z$, are
independent copies of the Markov-chain $m\mapsto\eta_{\pm}(m)$, starting with
initial conditions $\eta_{j,\pm}(0)=0$.
\end{enumerate}
\end{proposition}

The statement is intuitively clear. The mathematical content of the driving
rules \eqref{walktransprob} of the random walk $X(n)$ is exactly this: whenever
the walk visits a site $j\in\Z$, the probability of jumping to the left or to
the right (i.e.\ to site $j-1$ or to site $j+1$), conditionally on the whole
past, will depend only on the difference of the number of past jumps from $j$
to $j-1$, respectively, from $j$ to $j+1$, and independent of what had happened
at other sites.  The  more lengthy formal proof goes through exactly the same
steps as the corresponding statement in  \cite{toth_95}. We omit here the
formal proof.

Fix now $j\in\Z_-$ and $r\in\N$. The definitions
\eqref{Ldef}, \eqref{gamdef}, \eqref{xijdef},
\eqref{taujdef} and \eqref{etajdef} imply that
\begin{align}
\label{Lrepr1}
&
L_{j,r}(j)=r
\\[8pt]
\label{Lrepr2}
&
L_{j,r}(k+1)=L_{j,r}(k)+1+\eta_{k+1,-}(L_{j,r}(k)+1),&
j
&\le
k<0,
\\[8pt]
\label{Lrepr3}
&
L_{j,r}(k+1)=L_{j,r}(k)+\eta_{k+1,-}(L_{j,r}(k)),
&
0
&\le
k<\infty,
\\[8pt]
\label{Lrepr4}
&
L_{j,r}(k-1)=L_{j,r}(k)+\eta_{k,+}(L_{j,r}(k)),
&
-\infty
&<
k\le j.
\end{align}
Similar formulas are found for $j\in\Z_+$ and $r\in\N$.

Note that if
$L_{j,r}(k_0)=0$ for some $k_0\ge0$ (respectively, for some $k_0\le j$) then
$L_{j,r}(k)=0$ for all $k\ge k_0$ (respectively, for all $k\le k_0$).

The idea of the further steps of proof can be summarized in terms of
the above setup. With fixed
$x\in\R_-$ and $h\in\R_+$,
we choose
$j=\lfloor Ax\rfloor$ and $r=\lfloor Ah\rfloor$
with the scaling parameter
$A\to\infty$
at the end. We know from Lemma \ref{expconvlemma} that the
Markov-chains
$\eta_{j,\pm}$
converge
exponentially fast to their stationary distribution $\rho$.
This allows us to couple efficiently the increments
$L_{\lfloor Ax\rfloor,\lfloor Ah\rfloor}(k+1)-L_{\lfloor Ax\rfloor,\lfloor
Ah\rfloor}(k)$
with properly chosen i.i.d.\ random variables as long as the value of
$L_{\lfloor Ax\rfloor,\lfloor Ah\rfloor}(k)>A^{1/2+\vareps}$
and to use the law of large numbers.
This coupling does not apply when the value of
$L_{\lfloor Ax\rfloor,\lfloor Ah\rfloor}(k)<A^{1/2+\vareps}$.
We prove that once the value of
$L_{\lfloor Ax\rfloor,\lfloor Ah\rfloor}(k)$
drops below this threshold,
$L_{\lfloor Ax\rfloor,\lfloor Ah\rfloor}(k)$
hits zero (and sticks there) in
$o(A)$
time, with high probability. These steps of the proof are presented in
the next two subsections.

\subsection{Coupling}

We are in the context of the representation
\eqref{Lrepr1}, \eqref{Lrepr2}, \eqref{Lrepr3}, \eqref{Lrepr4} with
$j=\lfloor Ax\rfloor$, $r=\lfloor Ah\rfloor$.
Due to  Lemma \ref{expconvlemma}, we can realize jointly the
\emph{pairs of coupled processes}
\begin{equation}
\big(m\mapsto(\eta_{k,-}(m), \wt\eta_k(m))\big)_{k> j},
\qquad
\big(m\mapsto(\eta_{k,+}(m), \wt\eta_k(m))\big)_{k\le j}
\end{equation}
with the following properties.
\begin{enumerate}[--]

\item
The pairs of coupled processes with different $k$-indices are
independent.

\item
The processes $\big(m\mapsto\eta_{k,-}(m)\big)_{k>j}$ and
$\big(m\mapsto\eta_{k,+}(m)\big)_{k\le j}$ are those of the previous
subsection. I.e.\ they are independent copies of the Markov-chain
$m\mapsto\eta(m)$ with initial conditions $\eta_{k,\pm}(0)=0$.

\item
The processes
$\big(m\mapsto\wt\eta_{k}(m)\big)_{k\in\Z}$
are independent copies of the \emph{stationary} process
$m\mapsto\eta(m)$. I.e.\ these processes are initialized independently
with $\prob{\wt\eta_k(0)=x}=\rho(x)$ and run independently of one
another.

\item
The pairs of coupled processes
$m\mapsto(\eta_{k,\pm}(m), \wt\eta_k(m))$ are coalescing. This means
the following: we define the coalescence time
\begin{equation}
\mu_k:=\inf\{m\ge0:\eta_{k,\pm}(m)=\wt\eta_k(m)\}.
\end{equation}
Then, for $m\ge\mu_k$, the two processes stick together:
$\eta_{k,\pm}(m)=\wt\eta(m)$. Mind that the random variables
$\mu_k$, $k\in\Z$ are i.i.d.

\item
The tail of the distribution of the coalescence times decays
exponentially fast:
\begin{equation}
\label{coaltail}
\prob{\mu_k>m}<C e^{-\beta m}.
\end{equation}

\end{enumerate}

We define the processes $k\mapsto\wt L_{j,r}(k)$ similarly to the
processes $k\mapsto L_{j,r}(k)$ in
\eqref{Lrepr1}, \eqref{Lrepr2}, \eqref{Lrepr3}, \eqref{Lrepr4},
with the $\eta$-s replaced by the $\wt \eta$-s:
\begin{align*}
&
\wt L_{j,r}(j)=r
\\[8pt]
&
\wt L_{j,r}(k+1)=\wt L_{j,r}(k)+1+\wt\eta_{k+1,-}(\wt L_{j,r}(k)+1),
&
j
&\le
k<0,
\\[8pt]
&
\wt L_{j,r}(k+1)=\wt L_{j,r}(k)+\wt \eta_{k+1,-}(\wt L_{j,r}(k)),
&
0
&
\le
k<\infty,
\\[8pt]
&
\wt L_{j,r}(k-1)=\wt L_{j,r}(k)+\wt \eta_{k,+}(\wt L_{j,r}(k)),
&
-\infty
&<
k\le j.
\end{align*}
Note that the increments of this process are \emph{independent} with
distribution
\begin{align*}
&
\prob{\wt L_{j,r}(k+1)-\wt L_{j,r}(k)=z}=
\rho(z-1),
&
j
&\le
k<0,
\\[8pt]
&
\prob{\wt L_{j,r}(k+1)-\wt L_{j,r}(k)=z}=
\rho(z),
&
0
&\le
k<\infty,
\\[8pt]
&
\prob{\wt L_{j,r}(k-1)-\wt L_{j,r}(k)=z}=
\rho(z),
&
-\infty
&<
k\le j.
\end{align*}
Hence, from \eqref{mean1/2}, it follows that for any $K<\infty$
\begin{equation}
\label{coupledconv}
\sup_{|y|\le K}
\left|
A^{-1}\wt L_{\lfloor Ax\rfloor,\lfloor Ah\rfloor}(\lfloor Ay\rfloor)
-
\big((|x|-|y|)/2+h\big)
\right|
\toprob
0.
\end{equation}
Actually, by Doob's inequality, the following large deviation estimate
holds: for any $x\in\R$, $h\in\R_+$ and $K<\infty$ fixed
\begin{equation}
\label{largedev}
\prob{
\sup_{|k|\le AK}
\left|
\wt L_{\lfloor Ax\rfloor,\lfloor Ah\rfloor}(k)
-
\big((A|x|-|k|)/2+Ah\big)
\right|
>A^{1/2+\vareps}
}
<
C e^{-\beta A^{2\vareps}}.
\end{equation}
(The constants $C<\infty$ and $\beta>0$ do depend on the fixed
parameters $x$, $h$ and $K$.)
Denote now
\begin{align*}
&
\kappa_{j,r}^+
:=
\min\{k\ge j: L_{j,r}(k)\not=\wt L_{j,r}(k)\},
\\[8pt]
&
\kappa_{j,r}^-
:=
\max\{k\le j: L_{j,r}(k)\not=\wt L_{j,r}(k)\}.
\end{align*}
iThen, for $k\ge j$:
\begin{equation}
\label{kapp+diff}
\begin{split}
&
\prob{\kappa_{j,r}^+\le k+1}
-
\prob{\kappa_{j,r}^+ \le k}
=
\\[8pt]
&
\quad
=
\prob{\kappa_{j,r}^+ = k+1, \ \ \wt L_{j,r}(k)\le   A^{1/2+\vareps}}
+
\prob{\kappa_{j,r}^+ = k+1, \ \ \wt L_{j,r}(k)\ge   A^{1/2+\vareps}}
\\[8pt]
&
\quad
\le
\prob{\wt L_{j,r}(k)\le   A^{1/2+\vareps}}
+
\condprob{\kappa_{j,r}^+=k+1}
{\kappa_{j,r}^+>k, \ \ L_{j,r}(k)=\wt L_{j,r}(k)\ge
  A^{1/2+\vareps}}.
\end{split}
\end{equation}
Similarly, for $k\le j$:
\begin{equation}
\label{kapp-diff}
\begin{split}
&
\prob{\kappa_{j,r}^-\ge k-1}
-
\prob{\kappa_{j,r}^- \ge k}
=
\\[8pt]
&
\quad
=
\prob{\kappa_{j,r}^- = k-1, \ \ \wt L_{j,r}(k)\le   A^{1/2+\vareps}}
+
\prob{\kappa_{j,r}^- = k-1, \ \ \wt L_{j,r}(k)\ge   A^{1/2+\vareps}}
\\[8pt]
&
\quad
\le
\prob{\wt L_{j,r}(k)\le  A^{1/2+\vareps}}
+
\condprob{\kappa_{j,r}^-=k-1}
{\kappa_{j,r}^-<k, \ \ L_{j,r}(k)=\wt L_{j,r}(k)\ge
  A^{1/2+\vareps}}.
\end{split}
\end{equation}
Now, from \eqref{largedev}, it follows that  for
$|k| \le A(|x|+2h)-4A^{1/2+\vareps}$
\begin{equation}
\label{ldappl}
\prob{\wt L_{j,r}(k)\le   A^{1/2+\vareps}}
\le
C e^{-\beta A^{2\vareps}}.
\end{equation}
On the other hand, from \eqref{coaltail},
\begin{align}
\label{tail+appl}
&
\condprob{\kappa_{j,r}^+=k+1}
{\kappa_{j,r}^+>k, \ \ L_{j,r}(k)=\wt L_{j,r}(k)\ge   A^{1/2+\vareps}}
\le
C e^{-\beta   A^{1/2+\vareps}},
\\[8pt]
\label{tail-appl}
&
\condprob{\kappa_{j,r}^-=k-1}
{\kappa_{j,r}^-<k, \ \ L_{j,r}(k)=\wt L_{j,r}(k)\ge  A^{1/2+\vareps}}
\le
C e^{-\beta  A^{1/2+\vareps}}
\end{align}
with some constants $C<\infty$ and $\beta>0$, which do depend on all fixed
parameters and may vary from formula to formula.

Putting together \eqref{kapp+diff}, \eqref{ldappl}, \eqref{tail+appl},
respectively, \eqref{kapp-diff}, \eqref{ldappl}, \eqref{tail-appl}
and noting that $\prob{\kappa_{j,r}^+=j}=0$, we conclude that
\begin{align}
\label{stick}
&
\prob{
\min\big\{|k|:
L_{\lfloor Ax\rfloor,\lfloor Ah\rfloor}(k)
\not=
\wt L_{\lfloor Ax\rfloor,\lfloor Ah\rfloor}(k)
\big\}
\le
A(|x|+2h)-4A^{1/2+\vareps}
}
\le
C A e^{-\beta A^{2\vareps}},
\\[8pt]
\label{small}
&
\prob{L_{\lfloor Ax\rfloor,\lfloor Ah\rfloor}
(\pm\lfloor A(|x|+2h)-4A^{1/2+\vareps}\rfloor)
\ge 3 A^{1/2+\vareps}
}
\le
C  e^{-\beta A^{2\vareps}}.
\end{align}

\subsection{Hitting of $0$}
\label{ss:hit0}

It follows from Lemma \ref{expconvlemma}
that all moments of the distributions
$P^n(0,\cdot)$
converge to the corresponding moments  of
$\rho$.
In particular, for any $\delta>0$ there exists $n_\delta<\infty$, such
that
\begin{equation*}
\sum_{x\in\Z}P^n(0,x)x
\le
-\frac1{2+\delta}
\end{equation*}
holds if
$n\ge n_\delta$.

Consider now the Markov-chains defined by \eqref{Lrepr3} or \eqref{Lrepr4} (the
two are identical in law):
\begin{equation*}
L(k+1)=L(k)+\eta_{k+1}(L(k)),
\qquad
L(0)=r\in\N,
\end{equation*}
where $m\mapsto\eta_k(m)$, $k=1,2,3,\dots$ are i.i.d.\ copies of the
Markov-chain $m\mapsto\eta(m)$ with initial conditions $\eta_k(0)=0$. Define
the stopping times
\begin{equation*}
\tau_x:=\min\{k:L(k)\le x\},
\qquad
x=0,1,2,\dots.
\end{equation*}

\begin{lemma}
\label{lem:hitexpect}
For any $\delta>0$ there exists  $K_\delta<\infty$ such that for any
$r\in\N$:
\begin{equation*}
\condexpect{\tau_0}{L(0)=r}
\le
(2+\delta) r +K_\delta.
\end{equation*}

\begin{proof}
Clearly,
\begin{equation*}
\condexpect{\tau_0}{L(0)=r}
\le
\condexpect{\tau_{n_\delta}}{L(0)=r}
+
\max_{0\le s \le n_\delta}
\condexpect{\tau_0}{L(0)=s}.
\end{equation*}
Now, by optional stopping,
\begin{equation*}
\condexpect{\tau_{n_\delta}}{L(0)=r}
\le
(2+\delta) r,
\end{equation*}
and obviously,
\begin{equation*}
K_\delta:=
\max_{0\le s \le n_\delta}
\condexpect{\tau_0}{L(0)=s}
<\infty.
\end{equation*}
\end{proof}
\end{lemma}

In particular, choosing $\delta=1$ and applying  Markov's inequality,
it follows that
\begin{equation}
\label{hitright}
\begin{split}
&
\condprob{\rho^+_{\lfloor Ax\rfloor,\lfloor Ax\rfloor}>
 A(|x|+2h)+A^{1/2+2\vareps}}
{L_{\lfloor Ax\rfloor,\lfloor Ah\rfloor}
(\lfloor A(|x|+2h)-4A^{1/2+\vareps}\rfloor)
\le 3  A^{1/2+\vareps}}
\\[8pt]
&
\hskip3cm
\le
\frac{9  A^{1/2+\vareps} + K_1}{5A^{1/2+2\vareps}}
<
2 A^{-\vareps},
\end{split}
\end{equation}
and similarly,
\begin{equation}
\label{hitleft}
\begin{split}
&
\condprob{\lambda^+_{\lfloor Ax\rfloor,\lfloor Ax\rfloor}<
-A(|x|+2h)-A^{1/2+2\vareps}}
{L_{\lfloor Ax\rfloor,\lfloor Ah\rfloor}
(-\lfloor A(|x|+2h)+4A^{1/2+\vareps}\rfloor)
\le 3  A^{1/2+\vareps}}
\\[8pt]
&
\hskip3cm
\le
\frac{9  A^{1/2+\vareps} + K_1}{5A^{1/2+2\vareps}}
<
2 A^{-\vareps}.
\end{split}
\end{equation}
Eventually, Theorem \ref{thmSconv} follows from
\eqref{coupledconv},
\eqref{stick},
\eqref{small},
\eqref{hitright} and
\eqref{hitleft}.

\section{Proof of the theorem for the position of the random walker}
\label{s:procconvproof}

First, we introduce the following notations. For
$n\in\N$
and
$k\in\Z$,
let
\begin{equation*}
P(n,k)
:=
\prob{X(n)=k}
\end{equation*}
be the distribution of the position of the random walker. For
$s\in\R_+$,
\begin{equation}
\label{Rdef}
R(s,k)
:=
(1-e^{-s})\sum_{n=0}^\infty e^{-sn}P(n,k)
\end{equation}
is the distribution of
$X(\theta_s)$
where
$\theta_s$
has geometric distribution \eqref{thetadistr} and it is independent of
$X(n)$.

Also \eqref{piscaling} tells us that the proper definition of the
rescaled distribution is
\begin{equation*}
\varphi_A(t,x)
:=
A^{1/2}P(\lfloor At\rfloor,\lfloor A^{1/2}x\rfloor),
\end{equation*}
if
$t\in\R_+$
and
$x\in\R$.
Let
\begin{equation}
\label{pihatAdef}
\hat\varphi_A(s,x)
:=
A^{1/2}R(A^{-1}s,\lfloor A^{1/2}x\rfloor),
\end{equation}
which is asymptotically the Laplace-transform of
$\pi_A$
as
$A\to\infty$.

With these definitions, the statement of Theorem \ref{thmprocconv} is
equivalent to
\begin{equation*}
\hat\varphi_A(s,x)\to\hat\varphi(s,x),
\end{equation*}
which is proved below.

We will need the Laplace-transform
\begin{equation*}
\hat\rho(s,x,h)
=
s\,
\expect{\exp\left(-s\,\cT_{x,h}\right)}
=
se^{-s(|x|+2h)^2},
\end{equation*}
for which
\begin{equation*}
\hat\varphi(s,x)
=
2\int_0^\infty\hat\rho(s,|x|,h)\d h
\end{equation*}
holds.

\begin{proof}
[Proof of Theorem \ref{thmprocconv}]
Fix
$x\ge0$.
We can proceed in the case
$x\le0$
similarly. We start with the identity
\begin{equation}
\label{starteq}
P(n,k)
=
\prob{X_n=k}
=
\sum_{m=0}^\infty
\left(
\prob{T^+_{k-1,m}=n}
+
\prob{T^-_{k+1,m}=n}
\right),
\end{equation}
which is easy to check. From the definitions \eqref{Rdef} and
\eqref{pihatAdef},
\begin{equation}
\label{pihatA1}
\begin{split}
\hat\varphi_A(s,x)
&=
\frac{1-e^{-s/A}}{s/A}s\sum_{m=0}^\infty
\frac1{\sqrt A}e^{-ns/A}P(n,\lfloor A^{1/2}x\rfloor)
\\[8pt]
&=
\frac{1-e^{-s/A}}{s/A}s\sum_{m=0}^\infty \frac1{\sqrt A}
\left(\expect{e^{-s/A T^+_{\lfloor A^{1/2}x\rfloor-1,m}}}+\expect{e^{-s/A
T^-_{\lfloor A^{1/2}x\rfloor+1,m}}}\right),
\end{split}
\end{equation}
where we used \eqref{starteq} in the second equality. Let
\begin{equation*}
\hat\rho^\pm_A(s,x,h)
=
s\,\expect{\exp\big(-\frac sA T^\pm_{\lfloor A^{1/2}x\rfloor,\lfloor
A^{1/2}h\rfloor}\big)}.
\end{equation*}
Then \eqref{pihatA1} can be written as
\begin{equation}
\label{pihatA2}
\hat\varphi_A(s,x)=\frac{1-e^{-s/A}}{s/A}\int_0^\infty
\left(\hat\rho^+_A(s,x-A^{-1/2},h)+\hat\rho^-_A(s,x+A^{-1/2},h)\right)\d
h.
\end{equation}

It follows from \eqref{Tconv} that for all
$s>0$, $x\in\R$
and
$h>0$,
\begin{equation*}
\hat\rho^\pm_A(s,x,h)\to\hat\rho(s,x,h)
\end{equation*}
as
$A\to\infty$.
Applying Fatou's lemma in \eqref{pihatA2} yields
\begin{equation*}
\liminf_{A\to\infty}\hat\varphi_A(s,x)\ge2\int_0^\infty
\hat\rho(s,x,h)\d h
=
\hat\varphi(s,x).
\end{equation*}
If we use Fatou's lemma again, we get
\begin{equation*}
1=\int_{-\infty}^\infty \hat\varphi(s,x)\d x
\le
\int_{-\infty}^\infty \liminf_{A\to\infty}\hat\varphi_A(s,x)\d x
\le
\liminf_{A\to\infty} \int_{-\infty}^\infty \hat\varphi_A(s,x)\d x=1,
\end{equation*}
which gives for all $s\in\R$ that
\begin{equation}
\label{liminfpi}
\hat\varphi(s,x)=\liminf_{A\to\infty}\hat\varphi_A(s,x)
\end{equation}
holds for almost all
$x\in\R$.
Note that \eqref{liminfpi} is also true for any subsequence
$A_k\to\infty$,
which implies the assertion of Theorem
\ref{thmprocconv}.
\end{proof}

\section{Computer simulations}
\label{s:simulations}

We have prepared computer simulations with exponential weight functions
$w(k)=2^k$
and
$w(k)=10^k$.

Note that the limit objects in our theorems do not depend on the
choice of the weight function
$w$.
Therefore, we expect that the behaviour of the local time and the
trajectories is qualitatively similar, and we will find only
quantitative differences.

\begin{figure}
\centering
\includegraphics[width=200pt]{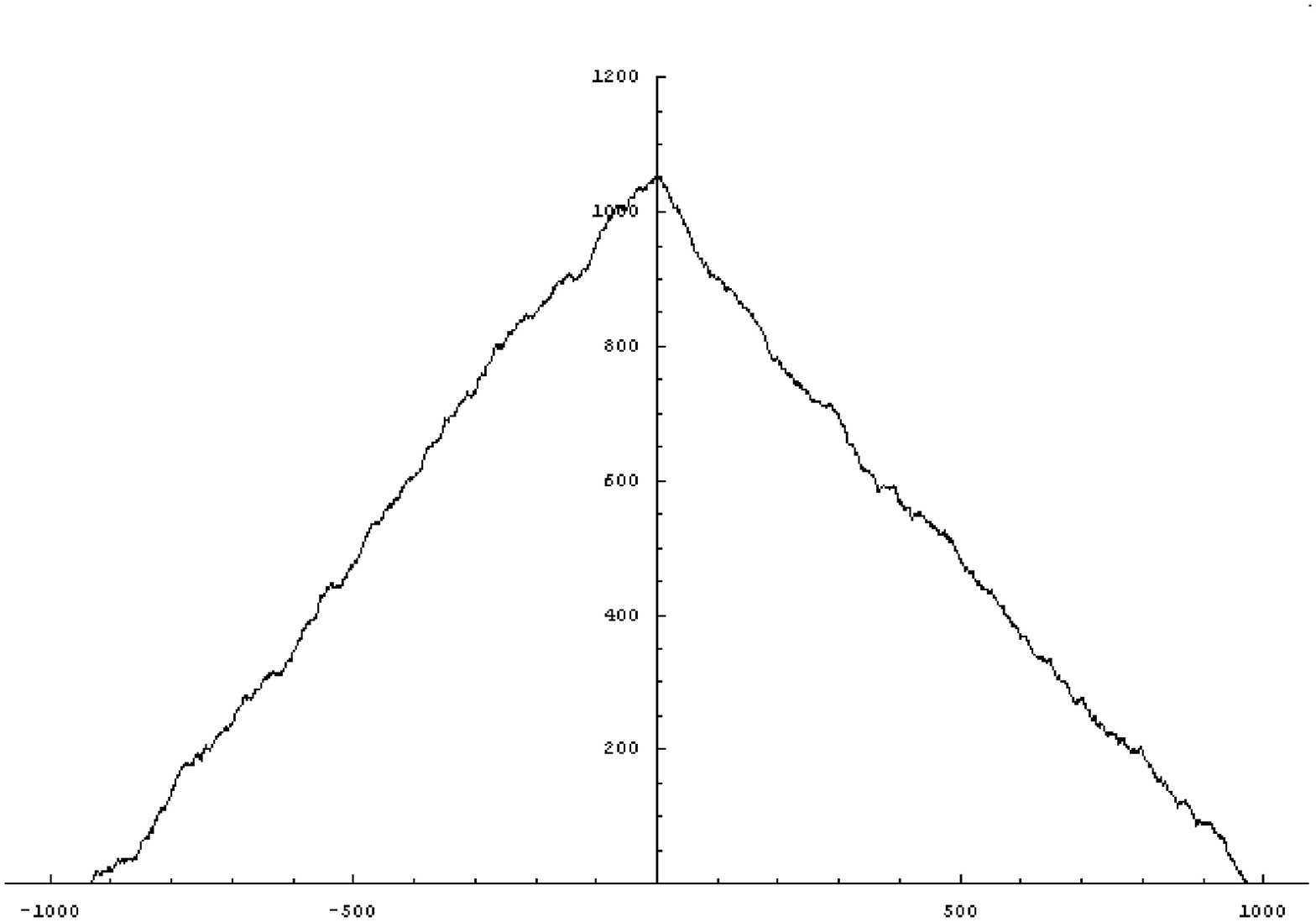}
\qquad
\includegraphics[width=200pt]{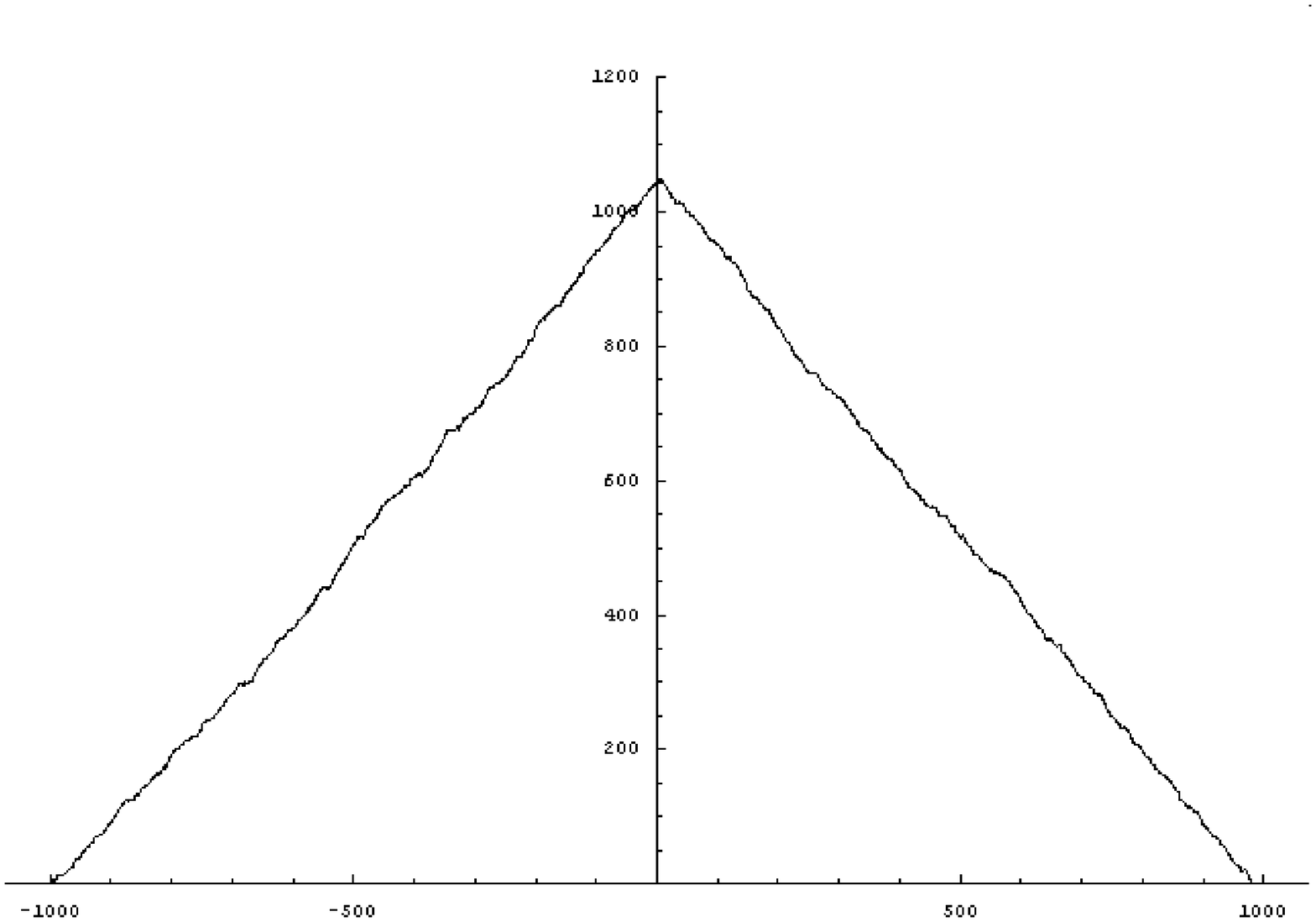}
\caption{The local time process of the random walk with
$w(k)=2^k$ and
$w(k)=10^k$
\label{loc}}
\end{figure}

Figure \ref{loc} shows the local time process of the random walk after
approximately
$10^6$
steps. More precisely, we have plotted the value of
$\Lambda^+_{100,800}$
with
$w(k)=2^k$
and
$w(k)=10^k$
respectively. One can see that the limits are the same in the two
cases -- according to Theorem \ref{thmSconv} -- but the rate of
convergence does depend on the choice of
$w$.
We can conclude the empirical rule that the faster the weight function
grows at infinity, the faster the convergence of the local time
process is.

\begin{figure}
\centering
\includegraphics[width=200pt]{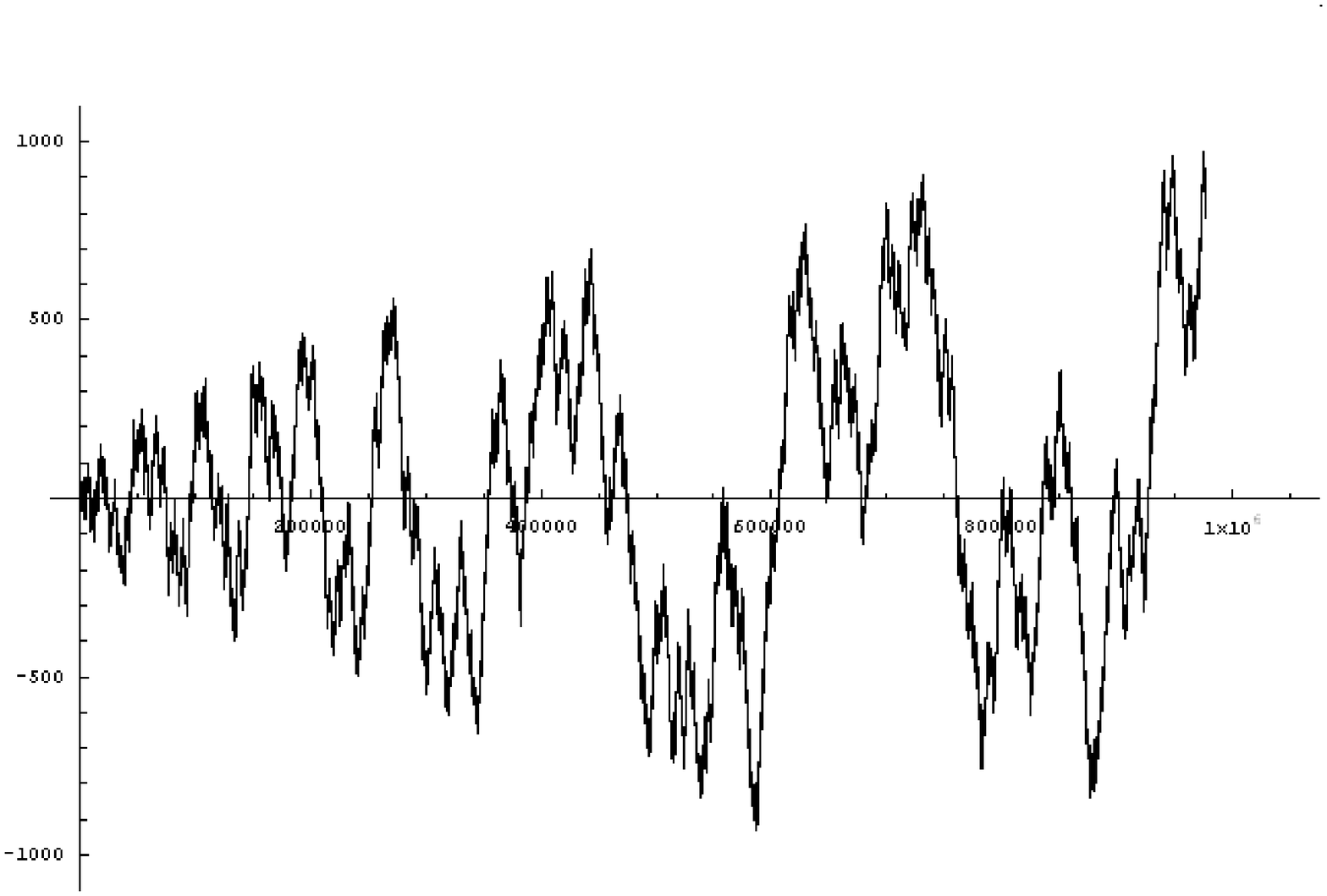}
\qquad
\includegraphics[width=200pt]{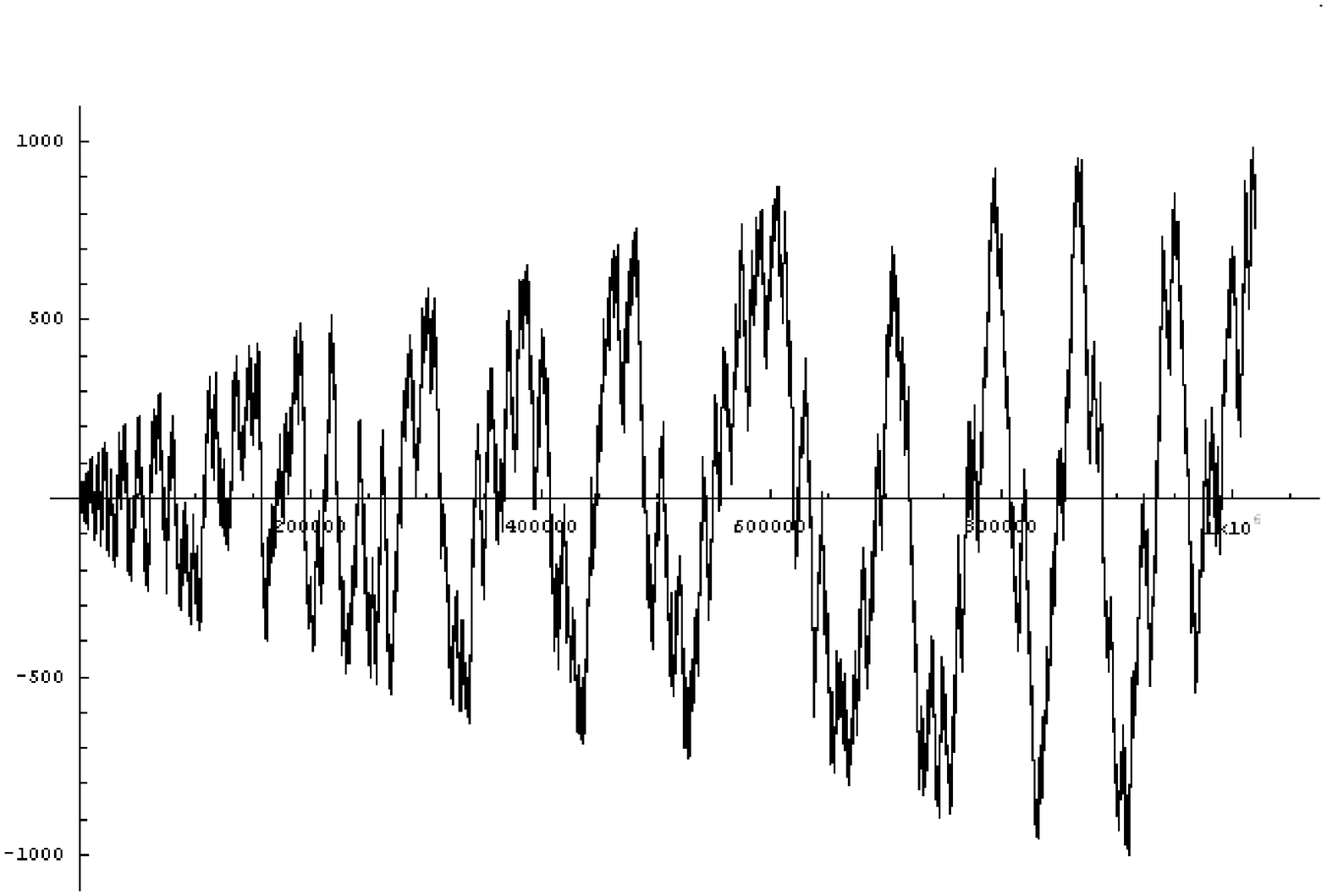}
\caption{The trajectories of the random walk with
$w(k)=2^k$
and
$w(k)=10^k$
\label{traj}}
\end{figure}

The difference between the trajectories of random walks generated with
various weights is more conspicuous. On Figure \ref{traj}, the
trajectories of the walks with
$w(k)=2^k$
and
$w(k)=10^k$
are illustrated, respectively. The number of steps is random, it is
about
$10^6$.
The data comes from the same sample as that shown on Figure \ref{loc}.

The first thing that we can observe on Figure \ref{traj} is that the
trajectories draw a sharp upper and lower hull according to
$\sqrt t$
and
$-\sqrt t$,
which agrees with our expectations after \eqref{detlim}. On the  other
hand, the trajectories oscillate very heavily between their extreme
values, especially in the case
$w(k)=10^k$,
there are almost but not quite straight crossings from
$\sqrt t$
to
$-\sqrt t$
and back. It shows that there is no continuous scaling limit of the
self-repelling random walk with directed edges.

The shape of the trajectories are slightly different in the cases
$w(k)=2^k$
and
$w(k)=10^k$.
The latter has heavier oscillations, because it corresponds to a
higher rate of growth of the weight function. Note that despite this
difference in the oscillation, the large scale behaviour is the same
on the two pictures on Figure \ref{traj}. The reason for this is that
if the random walk explores a new region, e.g.\ it exceeds its earlier
maximum, then the probability of the reversal does not depend on
$w$,
since the both outgoing edges have local time
$0$.
It can be a heuristic argument, why the upper and lower hulls
$\sqrt t$
and
$-\sqrt t$
are universal.

\vskip1cm

\subsubsection*{Acknowledgement}
We thank P\'eter M\'ora for his help in the computer simulations. The
research work of the authors is partially supported by the following OTKA
(Hungarian National Research Fund) grants: K 60708 (for B.T.\ and
B.V.), TS 49835 (for B.V.). B.V.\ thanks the kind hospitality of the Erwin
Schr\"odinger Institute (Vienna) where part of this work was done.

\bibliography{hazteto_bib}

\begin{thebibliography}{10}

\bibitem{amit_parisi_peliti_83}
D.~Amit, G.~Parisi, and L.~Peliti.
\newblock Asymptotic behaviour of the `true' self-avoiding walk.
\newblock {\em Phys. Rev. B}, 27:1635--1645, 1983.

\bibitem{cranston_mountford_96}
M.~Cranston and T.M. Mountford.
\newblock The strong law of large numbers for a brownian polymer.
\newblock {\em The Annals of Probability}, 24:1300--1323, 1996.

\bibitem{durrett_rogers_92}
R.T. Durrett and L.C.G. Rogers.
\newblock Asymptotic behaviour of a brownian polymer.
\newblock {\em Probab. Theory Relat. Fields}, 92:337--349, 1992.

\bibitem{mountford_tarres_08}
T.S. Mountford and P.~Tarr{\'e}s.
\newblock {\em Annales de l'Institut Henri Poincar\'e -- Probabilit\'es et
  Statistiques}, 44:29--46, 2008.

\bibitem{norris_rogers_williams_87}
J.R. Norris, L.C.G. Rogers, and D.~Williams.
\newblock Self-avoiding random walk: a brownian motion model with local time
  drift.
\newblock {\em Probability Theory and Related Fields}, 74:271--287, 1987.

\bibitem{nummelin_84}
E.~Nummelin.
\newblock {\em Irreducible Markov Chains and Non-Negative Operators}.
\newblock Cambridge University Press, 1984.

\bibitem{obukhov_peliti_83}
S.P. Obukhov and L.~Peliti.
\newblock Renormalisation of the "true" self-avoiding walk.
\newblock {\em J. Phys. A,}, 16:L147--L151, 1983.

\bibitem{peliti_pietronero_87}
L.~Peliti and L.~Pietronero.
\newblock Random walks with memory.
\newblock {\em Riv. Nuovo Cimento}, 10:1--33, 1987.

\bibitem{pemantle_07}
R.~Pemantle.
\newblock A survey of random processes with reinforcement.
\newblock {\em Probabolity Surveys}, 4:1--79, 2007.

\bibitem{toth_95}
B.~T{\'o}th.
\newblock The `true' self-avoiding walk with bond repulsion on {$\mathbb Z$}:
  limit theorems.
\newblock {\em Ann. Probab.}, 23:1523--1556, 1995.

\bibitem{toth_97}
B.~T{\'o}th.
\newblock Limit theorems for weakly reinforced random walks.
\newblock {\em Studia Sci. Math. Hungar.}, 33:321--337, 1997.

\bibitem{toth_99}
B.~T\'oth.
\newblock Self-interacting random motions -- a survey.
\newblock In P.~R{\'e}v{\'e}sz and B.~T{\'o}th, editors, {\em Random {W}alks},
  volume~9 of {\em Bolyai Society Mathematical Studies}, pages 349--384.
  J{\'a}nos Bolyai Mathematical Society, Budapest, 1999.

\bibitem{toth_01a}
B.~T\'oth.
\newblock Self-interacting random motions.
\newblock In {\em Proceedings of the 3rd European Congress of Mathematics,
  Barcelona}, volume~1, pages 555--565. Birkh{\"a}user, Boston-Basel-Berlin,
  2001.

\bibitem{toth_veto_08b}
B.~T{\'o}th and B.~Vet{\H o}.
\newblock Continuous time 'true' self-avoiding walk with site repulsion on
  {$\mathbb Z$}.
\newblock {\em in preparation}, 2008.

\bibitem{toth_werner_98}
B.~T{\'o}th and W.~Werner.
\newblock The true self-repelling motion.
\newblock {\em Probab. Theory Relat. Fields}, 111:375--452, 1998.

\end{thebibliography}
\bibliographystyle{plain}

\vfill \hbox{ \phantom{.}\hskip6cm \vbox{\hsize=20em

\noindent Address of authors:\\[5pt]
Institute of Mathematics\\
Budapest University of Technology\\
Egry J\'ozsef u. 1\\
H-1111 Budapest, Hungary\\[5pt]
e-mail:\\
{\tt balint@math.bme.hu\\vetob@math.bme.hu} }}

\end{document}